\documentclass{article}
\usepackage{amssymb}
\usepackage{amsmath}
\usepackage{amsthm}
\usepackage{enumitem}
\usepackage[left=3.5cm, right=3.5cm, top=2.5cm, bottom=3cm]{geometry}
\usepackage[mathscr]{euscript}
\title{Fixed Points of  Integral Type Contractions in Uniform Spaces with a Graph\\[0.3cm]}
\author{{Aris Aghanians$^1$\,\,\,and\,\,\,Kourosh Nourouzi$^1$\thanks{Corresponding author}
\thanks{e-mail: nourouzi@kntu.ac.ir; fax: +98 21 22853650}}\\[0.4cm]
{\em $^1$Department of Mathematics, K. N. Toosi University of Technology,}\\
{\em P.O. Box 16315-1618, Tehran, Iran}}
\theoremstyle{plain}
\newtheorem{lem}{Lemma}
\newtheorem{cor}{Corollary}
\newtheorem{thm}{Theorem}

\theoremstyle{definition}
\newtheorem{defn}{Definition}
\newtheorem{exm}{Example}
\theoremstyle{remark}
\newtheorem{rem}{Remark}

\newcommand\myint[1]{\int_0^{#1}\varphi(t){\rm d}t}
\DeclareMathOperator{\fix}{Fix}
\begin{document}
\maketitle \begin{abstract}{In this paper, we discuss the existence of fixed points for integral type contractions in uniform spaces endowed with both a graph and an $E$-distance. We also give two sufficient conditions under which the fixed point is unique. Our main results generalize some recent metric fixed point theorems.}
\end{abstract}
\def\thefootnote{ \ }
\footnotetext{{\em} $2010$ Mathematics Subject Classification.
47H10, 05C40.}

\section{Introduction and Preliminaries}
In \cite{bra}, Branciari discussed the existence and uniqueness of fixed points for mappings from a complete metric space $(X,d)$ into itself satisfying a general contractive condition of integral type. The result therein is a generalization of the Banach contraction principle in metric spaces. In fact, Branciari considered mappings $T:(X,d)\to(X,d)$ satisfying
$$\myint{d(Tx,Ty)}\leq\alpha\myint{d(x,y)}\qquad(x,y\in X),$$
where $\alpha\in(0,1)$ and $\varphi:[0,+\infty)\to[0,+\infty)$ is a Lebesgue-integrable function on $[0,+\infty)$ whose Lebesgue-integral is finite on each compact subset of $[0,+\infty)$, and satisfies $\myint{\varepsilon}>0$ for all $\varepsilon>0$. Recently, an integral version of  \'Ciri\'c's contraction was   given in  \cite{sam}.\par
In 2008, Jachymski \cite{jac} generalized the Banach contraction principle in metric spaces endowed with a graph. This idea was followed  by the authors (see \cite{knpan,knfpt}) in uniform spaces. In \cite{aam1}, the concept of an $E$-distance was introduced in uniform spaces as a generalization of a metric and a $w$-distance and then many different nonlinear contractions were generalized from metric to uniform spaces (see, e.g., \cite{agar,knir,ola}).\par
The aim of this paper is to study the existence and uniqueness of a fixed point for integral type contractions in uniform spaces endowed with both a graph and an $E$-distance. Our results generalize Theorem 2.1 in \cite{bra} as well as Corollary 3.1 in \cite{jac} by replacing metric spaces with uniform spaces endowed with a graph and by considering a weaker contractive condition. We also prove an integral version of \cite[Theorems 3.2 and 3.3]{jac}.\par
We begin with notions in uniform spaces that are needed in this paper. For more detailed discussion, the reader is referred to, e.g., \cite{wil}.\par
By a uniform space $(X,\mathscr U)$, shortly denoted here by $X$, it is meant a nonempty set $X$ together with a uniformity $\mathscr U$. For instance, if $d$ is a metric on a nonempty set $X$, then it induces a uniformity, called the uniformity induced by the metric $d$, in which the members of $\mathscr U$ are all the supersets of the sets
$$\big\{(x,y)\in X\times X:d(x,y)<\varepsilon\big\},$$
where $\varepsilon>0$.\par
It is well-known that a uniformity $\mathscr U$ on a nonempty set $X$ is separating if the intersection of all members of $\mathscr U$ is equal to the diagonal of the Cartesian product $X\times X$, that is, the set $\{(x,x):x\in X\}$ which is often denoted by $\Delta(X)$. If $\mathscr U$ is a separating uniformity on a nonempty set $X$, then the uniform space $X$ is said to be separated.\par
We next recall the definition of an $E$-distance on a uniform space $X$ as well as the notions of convergence, Cauchyness and completeness with $E$-distances.

\begin{defn}[\cite{aam1}] Let $X$ be a uniform space. A function $p:X\times X\rightarrow[0,+\infty)$ is called an $E$-distance on $X$ if
\begin{enumerate}[label={\roman*)}]
\item for each member $V$ of $\mathscr U$, there exists a $\delta>0$ such that $p(z,x)\leq\delta$ and $p(z,y)\leq\delta$ imply $(x,y)\in V$ for all $x,y,z\in X$;
\item the triangular inequality holds for $p$, that is,
$$p(x,y)\leq p(x,z)+p(z,y)\qquad(x,y,z\in X).$$
\end{enumerate}
\end{defn}

Let $p$ be an $E$-distance on a uniform space $X$. A sequence $\{x_n\}$ in $X$ is said to be $p$-convergent to a point $x\in X$, denoted by $x_n\stackrel{p}\longrightarrow x$, if it satisfies the usual metric condition, that is, $p(x_n,x)\rightarrow0$ as $n\rightarrow\infty$, and similarly, $p$-Cauchy if it satisfies $p(x_m,x_n)\rightarrow0$ as $m,n\rightarrow\infty$. The uniform space $X$ is called $p$-complete if every $p$-Cauchy sequence in $X$ is $p$-convergent to some point of $X$.\par
In the next lemma, an important property of $E$-distances in separated uniform spaces is formulated.

\begin{lem}[\cite{aam1}]\label{1}
Let $p$ be an $E$-distance on a separated uniform space $X$ and $\{x_n\}$ and $\{y_n\}$ be two arbitrary sequences in $X$. If $x_n\stackrel{p}\longrightarrow x$ and $x_n\stackrel{p}\longrightarrow y$, then $x=y$. In particular, if $x,y\in X$ and $p(z,x)=p(z,y)=0$ for some $z\in X$, then $x=y$.
\end{lem}

Finally, we recall some concepts about graphs. For more details on graph theory, see, e.g., \cite{bon}.\par
Let $X$ be a uniform space and consider a directed graph $G$ without any parallel edges such that the set $V(G)$ of its vertices is $X$, that is, $V(G)=X$ and the set $E(G)$ of its edges contains all loops, that is, $E(G)\supseteq\Delta(X)$. So the graph $G$ can be simply denoted by $G=(V(G),E(G))$. By $\widetilde G$, it is meant the undirected graph obtained from $G$ by ignoring the direction of the edges of $G$, that is,
$$V(\widetilde G)=X\quad\text{and}\quad E(\widetilde G)=\big\{(x,y)\in X\times X:\text{either}\ (x,y)\ \text{or}\ (y,x)\ \text{belongs to}\ E(G)\big\}.$$\par
A subgraph $H$ of $G$ is itself a directed graph such that $V(H)$ and $E(H)$ are contained in $V(G)$ and $E(G)$, respectively, and $(x,y)\in E(H)$ implies $x,y\in V(H)$ for all $x,y\in X$.\par
We need also a few notions about connectivity of graphs. Suppose that $x$ and $y$ are two vertices in $V(G)$. A finite sequence $(x_i)_{i=0}^N$ consisting of $N+1$ vertices of $G$ is a path in $G$ from $x$ to $y$ if $x_0=x$, $x_N=y$ and $(x_{i-1},x_i)\in E(G)$ for $i=1,\ldots,N$. The graph $G$ is weakly connected if there exists a path in $\widetilde G$ between each two vertices of $\widetilde G$.

\section{Main Results}
In this section, we consider the Euclidean metric on $[0,+\infty)$ and denote by $\lambda$ the Lebesgue measure on the Borel $\sigma$-algebra of $[0,+\infty)$. For a Borel set  $E=[a,b]$, we will use the notation $\int_a^b\varphi(t){\rm d}t$ to show the Lebesgue integral of a function $\varphi$ on $E$. We employ a class $\Phi$ consisting of all functions $\varphi:[0,+\infty)\rightarrow[0,+\infty)$ satisfying the following properties:
\begin{enumerate}[label={$(\Phi\arabic*)$}]
\item $\varphi$ is Lebesgue-integrable on $[0,+\infty)$;
\item The value of the Lebesgue integral $\myint{\varepsilon}$ is positive and finite for all $\varepsilon>0.$
\end{enumerate}\par
The next lemma embodies some important properties of functions of the class $\Phi$ which we need in the sequel.

\begin{lem}\label{phi}
Let $\varphi:[0,+\infty)\to[0,+\infty)$ be a function in the class $\Phi$ and $\{a_n\}$ be a sequence of nonnegative real numbers. Then the following statements hold:
\begin{enumerate}[label={\rm\arabic*.}]
\item If $\myint{a_n}\rightarrow0$ as $n\rightarrow\infty$, then $a_n\rightarrow0$ as $n\rightarrow\infty$.
\item If $\{a_n\}$ is monotone and converges to some $a\geq0$, then $\myint{a_n}\rightarrow\myint a$ as $n\rightarrow\infty$.
\end{enumerate}
\end{lem}

\begin{proof}
1. Let $\myint{a_n}\rightarrow0$ and suppose first on the contrary that $\limsup_{n\rightarrow\infty}a_n=\infty$. Then $\{a_n\}$ contains a subsequence $\{a_{n_k}\}$ which diverges to $\infty$. By passing to a subsequence if necessary, one may assume without loss of generality that $\{a_{n_k}\}$ is a nondecreasing subsequence of $\{a_n\}$. Because the sequence $\{\myint{a_{n_k}}\}$ of nonnegative numbers increases to zero,  so $a_{n_k}=0$ for all $k\geq1$. This is a  contradiction and therefore the sequence $\{a_n\}$ is bounded.\par
Next, if $\limsup_{n\rightarrow\infty}a_n=\varepsilon>0$, then there exists a strictly increasing sequence $\{n_k\}$ of positive integers such that $a_{n_k}\rightarrow\varepsilon$. Pick an integer $k_0>0$ so that the strict inequality $a_{n_k}>\frac\varepsilon2$ holds for all $k\geq k_0$. Therefore,
$$0<\myint{\frac\varepsilon2}\leq\myint{a_{n_k}}\rightarrow0,$$
which is again a contradiction. So $\limsup_{n\rightarrow\infty}a_n=0$, and consequently,
$$0\leq\liminf_{n\rightarrow\infty}a_n\leq\limsup_{n\rightarrow\infty}a_n=0,$$
that is, $a_n\rightarrow0$.\\
2. Let $\{a_n\}$ be nondecreasing and put $E_n=[0,a_n]$ for all $n\geq1$. Then each $E_n$ is a Borel subset of $[0,+\infty)$ and we have $E_1\subseteq E_2\subseteq\cdots$ and $\bigcup_{n=1}^\infty E_n=[0,a]$. Because the function $E\stackrel{\mu}\longmapsto\int_E\varphi{\rm d}\lambda$ is a Borel measure on $[0,+\infty)$, using the continuity of $\mu$ from below we get
$$\myint a=\mu\Big(\bigcup_{n=1}^\infty E_n\Big)=\lim_{n\rightarrow\infty}\mu(E_n)=\lim_{n\rightarrow\infty}\myint{a_n}.$$
A similar argument is true if $\{a_n\}$ is nonincreasing since each $E_n$ defined above is of finite $\mu$-measure by $(\Phi2)$.
\end{proof}

Let $T$ be a mapping from a uniform space $X$ endowed with a graph $G$ into itself. We denote as usual the set of all fixed points for $T$ by $\fix(T)$, and by $C_T$, we mean the set of all $x\in X$ such that $(T^nx,T^mx)$ is an edge of $\widetilde G$ for all $m,n\geq0$. Clearly, $\fix(T)\subseteq C_T$.

\begin{defn}
Let $p$ be an $E$-distance on a uniform space $X$ endowed with a graph $G$. We say that a mapping $T:X\to X$ is an integral type $p$-$G$-contraction if
\begin{enumerate}[label={IC\arabic*)}]
\item $T$ preserves the egdes of $G$, that is, $(x,y)\in E(G)$ implies $(Tx,Ty)\in E(G)$ for all $x,y\in X$;
\item there exists a $\varphi\in\Phi$ and a constant $\alpha\in(0,1)$ such that the contractive condition
$$\myint{p(Tx,Ty)}\leq\alpha\myint{p(x,y)}$$
holds for all $x,y\in X$ with $(x,y)\in E(G)$.
\end{enumerate}
\end{defn}

Now, we give some examples of integral type $p$-$G$-contractions.

\begin{exm}
Let $p$ be an $E$-distance on a uniform space $X$ endowed with a graph $G$ and $x_0$ be a point in $X$ such that $p(x_0,x_0)=0$. Since $E(G)$ contains the loop $(x_0,x_0)$, it follows that the constant mapping $T=x_0$ preserves the edges of $G$, and since $p(x_0,x_0)=0$, (IC2) holds trivially for any arbitrary $\varphi\in\Phi$ and $\alpha\in(0,1)$. Therefore, $T$ is an integral type $p$-$G$-contraction. In particular, each constant mapping on $X$ is an integral type $p$-$G$-contraction if and only if $p(x,x)=0$ for all $x\in X$.
\end{exm}

\begin{exm}
Let $(X,d)$ be a metric space and $T:X\to X$  a mapping satisfying
$$\myint{d(Tx,Ty)}\leq\alpha\myint{d(x,y)}\qquad(x,y\in X),$$
where $\varphi\in\Phi$ and $\alpha\in(0,1)$. If we consider $X$ as a uniform space with the uniformity induced by the metric $d$, then $T$ is an integral type $d$-$G_0$-contraction, where $G_0$ is the complete graph with the vertices set $X$, that is, $V(G_0)=X$ and $E(G_0)=X\times X$. The existence and uniqueness of fixed point for these kind of contractions were considered by Branciari in \cite{bra}.
\end{exm}

\begin{exm}
Let $\preceq$ and $p$ be a partial order and an $E$-distance on a uniform space $X$, respectively, and consider the poset graphs $G_1$ and $G_2$ by
$$V(G_1)=X\quad\text{and}\quad E(G_1)=\big\{(x,y)\in X\times X:x\preceq y\big\},$$
and
$$V(G_2)=X\quad\text{and}\quad E(G_2)=\big\{(x,y)\in X\times X:x\preceq y\vee y\preceq x\big\}.$$
Then integral type $p$-$G_1$-contractions are precisely the ordered integral type $p$-contractions, that is, nondecreasing mappings $T:X\to X$ which satisfy (IC2) for all $x,y\in X$ with $x\preceq y$ and for some $\varphi\in\Phi$ and $\alpha\in(0,1)$. And integral type $p$-$G_2$-contractions are those mappings $T:X\to X$ which are order preserving and satisfy (IC2) for all comparable $x,y\in X$ and for some $\varphi\in\Phi$ and $\alpha\in(0,1)$.
\end{exm}

\begin{rem}
Let $T$ be a mapping from an arbitrary uniform space $X$ into itself. If $X$ is endowed with the complete graph $G_0$, then the set $C_T$ coincides with $X$.\par
If $\preceq$ is a partial order on $X$ and $X$ is endowed with either $G_1$ or $G_2$, then a point $x\in X$ belongs to $C_T$ if and only if $T^nx$ is comparable to $T^mx$ for all $m,n\geq0$. In particular, if $T$ is monotone, then each $x\in X$ satisfying $x\preceq Tx$ or $Tx\preceq x$ belongs to $C_T$.
\end{rem}

\begin{exm}
Let $p$ be any arbitrary $E$-distance on a uniform space $X$ endowed with a graph $G$ and define a function $\varphi:[0,+\infty)\to[0,+\infty)$ by the rule $\varphi(t)=t^\beta$ for all $t\geq0$, where $\beta\geq0$ is constant. It is clear that $\varphi$ is Lebesgue-integrable on $[0,+\infty)$ and $\myint{\varepsilon}=\frac{\varepsilon^{1+\beta}}{1+\beta}$ which is positive and finite for all $\varepsilon>0$, that is, $\varphi\in\Phi$. Now, let a mapping $T:X\to X$ satisfy $p(Tx,Ty)\leq\alpha p(x,y)$ for all $x,y\in X$ with $(x,y)\in E(G)$, where $\alpha\in(0,1)$. Then $T$ satisfies (IC2) for the function $\varphi$ defined as above and the number $\alpha^{1+\beta}\in(0,1)$. In fact, if $x,y\in X$ and $(x,y)\in E(G)$, then
$$\myint{p(Tx,Ty)}=\frac{p(Tx,Ty)^{1+\beta}}{1+\beta}\leq\alpha^{1+\beta}\cdot\frac{p(x,y)^{1+\beta}}{1+\beta}=\alpha^{1+\beta}\myint{p(x,y)}.$$
Therefore, our contraction generalizes Banach's contraction with $E$-distances in uncountably many ways. In particular,
if $T$ is a Banach $G$-$p$-contraction (i.e., the Banach contraction in uniform spaces endowed with an $E$-distance and a graph), then $T$ is an integral type $p$-$G$-contraction for uncountably many functions $\varphi\in\Phi$.
\end{exm}

To prove the existence of a fixed point for an integral type $p$-$\widetilde G$-contraction, we need the following two lemmas:

\begin{lem}\label{step1}
Let $p$ be an $E$-distance on a uniform space $X$ endowed with a graph $G$ and $T:X\to X$ be an integral type $p$-$G$-contraction. Then $p(T^nx,T^ny)\rightarrow0$ as $n\rightarrow\infty$, for all $x,y\in X$ with $(x,y)\in E(G)$.
\end{lem}

\begin{proof}
Let $x,y\in X$ be such that $(x,y)\in E(G)$. According to Lemma \ref{phi}, it suffices to show that $\myint{p(T^nx,T^ny)}\rightarrow0$, where $\varphi\in\Phi$ is as in (IC2). To this end, note that because $T$ preserves the edges of $G$, we have $(T^nx,T^ny)\in E(G)$ for all $n\geq0$, and so by (IC2), we find
$$\myint{p(T^nx,T^ny)}\leq\alpha\myint{p(T^{n-1}x,T^{n-1}y)}\leq\cdots\leq\alpha^n\myint{p(x,y)}\qquad(n\geq1),$$
where $\alpha\in(0,1)$ is as in (IC2). Since, by $(\Phi2)$, $\myint{p(x,y)}$ is finite (even possibly zero), it follows immediately that $\myint{p(T^nx,T^ny)}\rightarrow0$.
\end{proof}

\begin{lem}\label{step2}
Let $p$ be an $E$-distance on a uniform space $X$ endowed with a graph $G$ and $T:X\to X$ be an integral type $p$-$\widetilde G$-contraction. Then the sequence $\{T^nx\}$ is $p$-Cauchy for all $x\in C_T$.
\end{lem}

\begin{proof}
Suppose on the contrary that $\{T^nx\}$ is not $p$-Cauchy for some $x\in C_T$. Then there exist an $\varepsilon>0$ and positive integers $m_k$ and $n_k$ such that
$$m_k>n_k\geq k\quad\text{and}\quad p(T^{m_k}x,T^{n_k}x)\geq\varepsilon\qquad k=1,2,\ldots\,.$$
If the integer $n_k$ is kept fixed for sufficiently large indices $k$ (say, $k\geq k_0$), then using Lemma \ref{step1}, one may assume without loss of generality that $m_k>n_k$ is the smallest integer with $p(T^{m_k}x,T^{n_k}x)\geq\varepsilon$, that is,
$$p(T^{m_k-1}x,T^{n_k}x)<\varepsilon\qquad(k\geq k_0).$$
Hence we have
\begin{eqnarray*}
\varepsilon&\leq&p(T^{m_k}x,T^{n_k}x)\cr
&\leq&p(T^{m_k}x,T^{m_k-1}x)+p(T^{m_k-1}x,T^{n_k}x)\cr
&<&p(T^{m_k}x,T^{m_k-1}x)+\varepsilon
\end{eqnarray*}
for each $k\geq k_0$. Since $x\in C_T$, it follows that $(Tx,x)\in E(\widetilde G)$ and by Lemma \ref{step1}, we have $p(T^{m_k}x,T^{m_k-1}x)\rightarrow0$. Thus, letting $k\rightarrow\infty$ yields $p(T^{m_k}x,T^{n_k}x)\rightarrow\varepsilon$. On the other hand, we have
$$p(T^{m_k+1}x,T^{n_k+1}x)\leq p(T^{m_k+1}x,T^{m_k}x)+p(T^{m_k}x,T^{n_k}x)+p(T^{n_k}x,T^{n_k+1}x)$$
for all $k\geq1$. Letting $k\rightarrow\infty$, since $(Tx,x),(x,Tx)\in E(\widetilde G)$, it follows by Lemma \ref{step1} that
$$\limsup_{k\rightarrow\infty}p(T^{m_k+1}x,T^{n_k+1}x)\leq\varepsilon.$$
Moreover, the inequality
$$p(T^{m_k+1}x,T^{n_k+1}x)\geq p(T^{m_k}x,T^{n_k}x)-p(T^{m_k}x,T^{m_k+1}x)-p(T^{n_k+1}x,T^{n_k}x)$$
holds for all $k\geq1$. Thus, similarly we have
$$\liminf_{k\rightarrow\infty}p(T^{m_k+1}x,T^{n_k+1}x)\geq\varepsilon.$$
Hence, $p(T^{m_k+1}x,T^{n_k+1}x)\rightarrow\varepsilon$. By passing to two subsequences with the same choice function if necessary, one may assume without loss of generality that both $\{p(T^{m_k}x,T^{n_k}x)\}$ and $\{p(T^{m_k+1}x,T^{n_k+1}x)\}$ are monotone. Therefore, using Lemma \ref{phi} twice, we have
$$\myint{\varepsilon}=\lim_{k\rightarrow\infty}\myint{p(T^{m_k+1}x,T^{n_k+1}x)} \leq\alpha\lim_{k\rightarrow\infty}\myint{p(T^{m_k}x,T^{n_k}x)}=\alpha\myint{\varepsilon},$$
where $\varphi\in\Phi$ and $\alpha\in(0,1)$ are as in (IC2). Therefore, $\myint{\varepsilon}=0$, which is a contradiction. Consequently, the sequence $\{T^nx\}$ is $p$-Cauchy for all $x\in C_T$.
\end{proof}

\begin{defn}
Let $p$ be an $E$-distance on a uniform space $X$ endowed with a graph $G$ and $T$ be a mapping from $X$ into itself. We say that
\begin{enumerate}[label={\roman*)}]
\item $T$ is orbitally $p$-$G$-continuous on $X$ if for all $x,y\in X$ and all sequences $\{a_n\}$ of positive integers with $(T^{a_n}x,T^{a_{n+1}}x)\in E(G)$ for $n=1,2,\ldots$, $T^{a_n}x\stackrel{p}\longrightarrow y$ as $n\rightarrow\infty$, implies $T(T^{a_n}x)\stackrel{p}\longrightarrow Ty$ as $n\rightarrow\infty$.
\item $T$ is a $p$-Picard operator if $T$ has a unique fixed point $u\in X$ and $T^nx\stackrel{p}\longrightarrow u$ for all $x\in X$.
\item $T$ is a weakly $p$-Picard operator if $\{T^nx\}$ is $p$-convergent to a fixed point of $T$ for all $x\in X$.
\end{enumerate}
\end{defn}

It is clear that each $p$-Picard operator is weakly $p$-Picard. Moreover, a weakly $p$-Picard operator is $p$-Picard if and only if its fixed point is unique.

\begin{exm}
Let $X$ be any arbitrary uniform space with more than one point equipped with an $E$-distance $p$. Choose a nonempty proper subset $A$ of $X$ and pick $a$ and $b$ from $A$ and $A^c$, respectively. Then the mapping $T:X\to X$ defined by $Tx=a$ if $x\in A$, and $Tx=b$ if $x\notin A$ is a weakly $p$-Picard operator which fails to be $p$-Picard. In fact, we have $\fix(T)=\{a,b\}$. Therefore, a weakly $p$-Picard operator is not necessarily $p$-Picard.
\end{exm}

Now, we are ready to prove our main theorems. The first result guarantees the existence of a fixed point when an integral type $p$-$\widetilde G$-contraction is orbitally $p$-$\widetilde G$-continuous on $X$ or the triple $(X,p,G)$ has a certain property.

\begin{thm}\label{contfix}
Let $p$ be an $E$-distance on a separated uniform space $X$ endowed with a graph $G$ such that $X$ is $p$-complete, and $T:X\to X$ be an integral type $p$-$\widetilde G$-contraction. Then $T\mid_{C_T}$ is a weakly $p$-Picard operator if one of the following statements holds:
\begin{enumerate}[label={\rm\roman*)}]
\item $T$ is orbitally $p$-$\widetilde G$-continuous on $X$;
\item The triple $(X,p,G)$ satisfies the following property:
\begin{itemize}[label={$(\ast)$}]
\item If a sequence $\{x_n\}$ in $X$ is $p$-convergent to an $x\in X$ and satisfies $(x_n,x_{n+1})\in E(\widetilde G)$ for all $n\geq1$, then there exists a subsequence $\{x_{n_k}\}$ of $\{x_n\}$ such that $(x_{n_k},x)\in E(\widetilde G)$ for all $k\geq1$.
\end{itemize}
\end{enumerate}
In particular, having been held {\rm(i)} or {\rm(ii)}, $\fix(T)\ne\emptyset$ if and only if $C_T\ne\emptyset$.
\end{thm}

\begin{proof}
If $C_T=\emptyset$, then there is nothing to prove. Otherwise, note first that since $T$ preserves the edges of $\widetilde G$, it follows that $C_T$ is $T$-invariant, that is, $T$ maps $C_T$ into itself. Now, let $x\in C_T$ be given. Then $(T^nx,T^{n+1}x)\in E(\widetilde G)$ for all $n\geq0$. Moreover, by Lemma \ref{step2}, the sequence $\{T^nx\}$ is $p$-Cauchy in $X$, and because $X$ is $p$-complete, there exists a $u\in X$ (depends on $x$) such that $T^nx\stackrel{p}\longrightarrow u$.\par
To prove the existence of a fixed point for $T$, suppose first that $T$ is orbitally $p$-$\widetilde G$-continuous. Then $T^{n+1}x\stackrel{p}\longrightarrow Tu$ and because $X$ is separated, Lemma \ref{1} ensures that $Tu=u$, that is, $u$ is a fixed point for $T$.\par
On the other hand, if Property $(\ast)$ holds, then $\{T^nx\}$ contains a subsequence $\{T^{n_k}x\}$ such that $(T^{n_k}x,u)\in E(\widetilde G)$ for all $k\geq1$. Since $p(T^{n_k}x,u)\rightarrow0$, by passing to a subsequence if necessary, one may assume without loss of generality that $\{p(T^{n_k}x,u)\}$ is monotone. Hence by Lemma \ref{phi}, we have
$$\myint{p(T^{{n_k}+1}x,Tu)}\leq\alpha\myint{p(T^{n_k}x,u)}\rightarrow0\quad\text{as}\quad k\rightarrow\infty,$$
where $\alpha\in(0,1)$ is as in (IC2). Using Lemma \ref{phi} once more, one obtains $p(T^{{n_k}+1}x,Tu)\rightarrow0$ and since $X$ is separated, Lemma \ref{1} guarantees that $Tu=u$, that is, $u$ is a fixed point for $T$.\par
Finally, $u\in\fix(T)\subseteq C_T$, and so $T\mid_{C_T}$ is a weakly $p$-Picard operator.
\end{proof}

Setting $G=G_0$ in Theorem \ref{contfix}, we have the following result, which is a generalization of \cite[Theorem 2.1]{bra} to uniform spaces equipped with an $E$-distance.

\begin{cor}
Let $p$ be an $E$-distance on a separated uniform space $X$ such that $X$ is $p$-complete. Let   $T:X\to X$ satisfy
$$\myint{p(Tx,Ty)}\leq\alpha\myint{p(x,y)}\qquad(x,y\in X),$$
where $\varphi\in\Phi$ and $\alpha\in(0,1)$. Then $T$ is a $p$-Picard operator.
\end{cor}

\begin{proof}
By Theorem \ref{contfix}, the mapping $T$ is a weakly $p$-Picard operator. To complete the proof, it suffices to show that $T$ has a unique fixed point. To this end, let $x$ and $y$ be two fixed points for $T$. Then
$$\myint{p(x,y)}=\myint{p(Tx,Ty)}\leq\alpha\myint{p(x,y)},$$
which is impossible unless $p(x,y)=0$. Similarly, one can show that $p(x,x)=0$ and since $X$ is separated, it follows by Lemma \ref{1} that $x=y$.
\end{proof}

Because $\widetilde{G_1}=\widetilde{G_2}=G_2$, setting $G=G_1$ or $G=G_2$ in Theorem \ref{contfix}, we obtain the ordered version of Branciari's result as follows:

\begin{cor}
Let $p$ be an $E$-distance on a partially ordered separated uniform space $X$ such that $X$ is $p$-complete and a mapping $T:X\to X$ satisfy
$$\myint{p(Tx,Ty)}\leq\alpha\myint{p(x,y)}$$
for all comparable elements $x$ and $y$ of $X$, where $\varphi\in\Phi$ and $\alpha\in(0,1)$. Assume that there exists an $x\in X$ such that $T^mx$ and $T^nx$ are comparable for all $m,n\geq0$. Then $T$ is a weakly $p$-Picard operator if one of the following statements holds:
\begin{itemize}[label={$-$}]
\item $T$ is orbitally $p$-$G_2$-continuous on $X$;
\item $X$ satisfies the following property:
\begin{itemize}
\item[] If a sequence $\{x_n\}$ in $X$ with successive comparable terms is $p$-convergent to an $x\in X$, then $x$ is comparable to $x_n$ for all $n\geq1$.
\end{itemize}
\end{itemize}
\end{cor}

Next, we are going to prove two theorems on uniqueness of the fixed points for integral type $p$-$\widetilde G$-contractions.

\begin{thm}
Let $p$ be an $E$-distance on a separated uniform space $X$ endowed with a graph $G$ such that $X$ is $p$-complete, and let $T:X\to X$ be an integral type $p$-$\widetilde G$-contraction such that the function $\varphi$ in (IC2) satisfies
\begin{equation}\label{subadd}
\myint{a+b}\leq\myint{a}+\myint{b}
\end{equation}
for all $a,b\geq0$. If $G$ is weakly connected and $C_T$ is nonempty, then there exists a unique $u\in X$ such that $T^nx\stackrel{p}\longrightarrow u$ for all $x\in X$. In particular, $T$ is a $p$-Picard operator if and only if $\fix(T)$ is nonempty.
\end{thm}

\begin{proof}
Let $x$ and $y$ be two arbitrary elements of $X$. Since $G$ is weakly connected, there exists a path $(x_i)_{i=0}^N$ in $\widetilde G$ from $x$ to $y$. Since $T$ preserves the edges of $\widetilde G$, it follows that $(T^nx_{i-1},T^nx_i)\in E(\widetilde G)$ for all $n\geq0$ and $i=1,\ldots,N$. Therefore, by \eqref{subadd} and (IC2) we have
\begin{eqnarray*}
\myint{p(T^nx,T^ny)}&\leq&\myint{\sum_{i=1}^Np(T^nx_{i-1},T^nx_i)}\cr\\
&\leq&\sum_{i=1}^N\myint{p(T^nx_{i-1},T^nx_i)}\cr\\
&\leq&\alpha\sum_{i=1}^N\myint{p(T^{n-1}x_{i-1},T^{n-1}x_i)}\cr\\
&\vdots&\cr\\[-5mm]
&\leq&\alpha^n\sum_{i=1}^N\myint{p(x_{i-1},x_i)}
\end{eqnarray*}
for all $n\geq0$, where $\varphi\in\Phi$ and $\alpha\in(0,1)$ are as in (IC2). Since, by $(\Phi2)$, $\sum_{i=1}^N\myint{p(x_{i-1},x_i)}$ is finite (possibly zero), it follows immediately that $\myint{p(T^nx,T^ny)}\rightarrow0$. Hence by Lemma \ref{phi}, $p(T^nx,T^ny)\rightarrow0$.\par
Now, pick a point $x\in C_T$. By Lemma \ref{step2}, the sequence $\{T^nx\}$ is $p$-Cauchy in $X$ and since $X$ is $p$-complete, there exists a $u\in X$ such that $T^nx\stackrel{p}\longrightarrow u$. If $y$ is an arbitrary point in $X$, then
$$0\leq p(T^ny,u)\leq p(T^ny,T^nx)+p(T^nx,u)\rightarrow0\quad\text{as}\quad n\rightarrow\infty.$$
So $T^ny\stackrel{p}\longrightarrow u$. The uniqueness of $u$ follows immediately from Lemma \ref{1}.
\end{proof}

\begin{thm}\label{thm4}
Let $p$ be an $E$-distance on a separated uniform space $X$ endowed with a graph $G$ and $T:X\to X$ be an integral type $p$-$\widetilde G$-contraction. If the subgraph of $G$ with the vertices $\fix(T)$ is weakly connected, then $T$ has at most one fixed point in $X$.
\end{thm}

\begin{proof}
Let $x$ and $y$ be two fixed points for $T$. Then there exists a path $(x_i)_{i=0}^N$ in $\widetilde G$ from $x$ to $y$ such that $x_1,\ldots,x_{N-1}\in\fix(T)$. Since $E(\widetilde G)$ contains all loops, we can assume without loss of generality that the length of this path, that is, the integer $N$ is even. Now, by (IC2) we have
$$\myint{p(x_{i-1},x_i)}=\myint{p(Tx_{i-1},Tx_i)}\leq\alpha\myint{p(x_{i-1},x_i)}\qquad i=1,\ldots,N,$$
where $\varphi\in\Phi$ and $\alpha\in(0,1)$, which is impossible unless $\myint{p(x_{i-1},x_i)}=0$ or equivalently, $p(x_{i-1},x_i)=0$ for $i=1,\ldots,N$. Because $E(\widetilde G)$ is symmetric, a similar argument yields $p(x_i,x_{i-1})=0$ for $i=1,\ldots,N$. Since $N$ is even, using Lemma \ref{1} finitely many times, we get $x=x_0=x_2=\cdots=x_N=y$. Consequently, $T$ has at most one fixed point in $X$.
\end{proof}

\begin{rem}
Theorem \ref{thm4} guarantees that in a separated uniform space $X$ endowed with a graph $G$ and an $E$-distance $p$, if $(x,y)\in E(G)$, then both $x$ and $y$ cannot be a fixed point for any integral type $p$-$\widetilde G$-contraction $T$. In other words, each weakly connected component of $G$ intersects $\fix(T)$ in at most one point. So in partially ordered separated uniform spaces equipped with an $E$-distance $p$, no ordered integral type $p$-contraction has two comparable fixed points.
\end{rem}

\begin{rem}
Since the Riemann integral (proper and improper) is subsumed in the Lebesgue integral, it follows that one may replace Lebesgue-integrability with Riemann-integrability of $\varphi$ on $[0,+\infty)$ in $(\Phi1)$, where the value of the integral on $[0,+\infty)$ is allowed to be $\infty$. Facing with Riemann integrals, we should assume that the function $\varphi$ is bounded. Therefore, all of the results of this paper can be restated and reproved for Riemann integrals instead of Lebesgue integrals. A similar remark holds for Riemann-Stieltjes integrable functions with respect to any fixed nondecreasing function on $[0,+\infty)$.
\end{rem}

\end{document}